\documentclass[a4paper,12pt,reqno]{amsart}

\usepackage[T1]{fontenc}
\usepackage[utf8]{inputenc}
\usepackage{lmodern}
\usepackage[english]{babel}

\usepackage[%
	left=2.5cm,       
	right=2.5cm,      
	top=3.5cm,        
	bottom=3.5cm,     
	heightrounded,    
	bindingoffset=0mm 
]{geometry}

\usepackage{amsmath}
\usepackage{amssymb}
\usepackage{mathtools}
\usepackage{mathrsfs}
\usepackage{soul}

\numberwithin{equation}{section}

\usepackage{hyperref} 
\usepackage[noabbrev,capitalize]{cleveref}

\usepackage{bookmark}

\theoremstyle{plain}
	\newtheorem{theorem}{Theorem}[section]
	\newtheorem{lemma}[theorem]{Lemma}

\theoremstyle{definition}
	\newtheorem{definition}[theorem]{Definition}
	\newtheorem{example}[theorem]{Example}
	
	\newtheorem{open.problem}[theorem]{Open Problem}

\usepackage[initials]{amsrefs}

\usepackage{tikz}
\usetikzlibrary{decorations.pathreplacing}

\usepackage[font=small,labelfont=bf]{caption}

\usepackage{float}

\usepackage{enumerate}
\usepackage{url}

\newcommand{\N}{\mathbb{N}}
\newcommand{\R}{\mathbb{R}}
\newcommand{\eps}{\varepsilon}
\newcommand{\de}{\partial}
\newcommand{\Haus}{\mathcal{H}}

\DeclareMathOperator{\dist}{dist}
\DeclareMathOperator{\diam}{diam}
\DeclareMathOperator{\co}{co}

\DeclarePairedDelimiter{\scalar}{<}{>}                                     
\DeclarePairedDelimiter{\set}{\{}{\}}

\DeclarePairedDelimiter{\ceil}{\lceil}{\rceil}

\allowdisplaybreaks

\mathchardef\ordinarycolon\mathcode`\:
\mathcode`\:=\string"8000
\begingroup \catcode`\:=\active
  \gdef:{\mathrel{\mathop\ordinarycolon}}
\endgroup

\begin{document}

\title{On the convex components of a set in $\R^n$}

\author[F.~Giannetti]{Flavia Giannetti}
\address[F.~Giannetti]{Dipartimento di Matematica ed Applicazioni ``R. Caccioppoli'',
Università degli Studi di Napoli ``Federico II'', Via Cintia, 80126 Napoli, Italy}
\email{giannett@unina.it}

\author[G.~Stefani]{Giorgio Stefani}
\address[G.~Stefani]{Scuola Internazionale Superiore di Studi Avanzati (SISSA), via Bonomea 265, 34136 Trieste, Italy}
\email{giorgio.stefani.math@gmail.com}

\date{\today}



\thanks{
\textit{Acknowledgements}.
The authors are  members of INdAM-GNAMPA. The first author
was partially supported by Università degli Studi di Napoli Federico II FRA Project 2020 \textit{Regolarità per minimi di funzionali ampiamente degeneri} (Project code: 000022) and by the INdAM--GNAMPA 2022 Project \textit{Enhancement e segmentazione immagini mediante operatori tipo campionamento e
metodi variazionali}, codice CUP\_E55\-F22\-00\-02\-70\-001.
The second author was partially supported by the ERC Starting Grant 676675 FLIRT -- \textit{Fluid Flows and Irregular Transport}, by the INdAM--GNAMPA 2020 Project \textit{Problemi isoperimetrici con anisotropie} (n.\ prot.\ U-UFMBAZ-2020-000798 15-04-2020), by the INdAM--GNAMPA 2022 Project \textit{Analisi geometrica in strutture subriemanniane}, codice CUP\_E55\-F22\-00\-02\-70\-001, and has received funding from the European Research Council (ERC) under the European Union’s Horizon 2020 research and innovation program (grant agreement No.~945655).
} 

\begin{abstract}
We prove a lower bound on the number of the convex components of a compact set with non-empty interior in $\R^n$ for all $n\ge2$.
Our result generalizes and improves  the inequalities previously obtained in~\cites{CGLP19,LL08}.
\end{abstract}

\maketitle

\section{Introduction}

\subsection{Convex components}
Let $n\ge2$. 
Let us consider a compact set $E\subset\R^n$ with non-empty interior  and a decomposition of the form
\begin{equation}
\label{eq:convex_decomp}
E=\bigcup_{i=1}^k E_i,
\end{equation}
where $k\in\N$ and $E_1,\dots,E_k$ are the \emph{convex components} of $E$, i.e., compact and convex sets with non-empty interior.
Since, in general, such a decomposition is obviously not unique, it is interesting to give a lower bound on the minimal number $k_{\min}(E)\in\N$ of the convex components of $E$. 
By definition, $k_{\min}(E)=1$ if and only if $E$ is a convex body.
Moreover, we can note that $k_{\min}(E)\ge c(E)$, where $c(E)\in\N$ is the number of connected components of $E$. Indeed, any convex component of $E$ must lay inside some connected component of $E$. Therefore, without loss of generality, in the following we will always assume that $E$ is a connected set.

The first lower bound on the minimal number of convex components was given in~\cite{LL08}*{Theorem~1.1}, where the authors proved that 
\begin{equation}\label{eq:leonetti_deficit}
k_{\min}(E)
\ge
\ceil*{
\frac{\Haus^{n-1}(\de E)}{\Haus^{n-1}(\de(\co(E)))}
},
\end{equation}
where $\ceil*{x}\in\mathbb Z$ denotes the upper integer part of $x\in\R$.
Here and in the following, for all $s\ge0$ we let $\Haus^s$ be the $s$-dimensional Hausdorff measure (in particular, $\Haus^0$ is the counting measure). 
Moreover, we let $\de E$ be the boundary of~$E$ and $\co(E)$ be the convex hull of~$E$.
Note that, since $E$ admits at least one decomposition as in~\eqref{eq:convex_decomp}, $\Haus^{n-1}(\de E)$ and $\Haus^{n-1}(\de(\co(E)))$ are two finite and strictly positive real numbers, see~\cite{LL08}, so that the right-hand side in~\eqref{eq:leonetti_deficit} is well defined.

In the subsequent paper \cite{CGLP19}, the bound in~\eqref{eq:leonetti_deficit} has been improved in the case $n=2$ in the sense explained in \cref{ss:improvement} for a class of compact sets $E\subset\R^2$.  
In the same spirit, our aim  is to provide a refined bound of the number $k_{\min}(E)$ in any dimension $n\ge 2$. 
We stress that the estimate we are going to obtain also improves the result in~\cite{CGLP19}.

\subsection{Monotonicity of perimeter}

The proof of~\eqref{eq:leonetti_deficit} is based on the following monotonicity property of the perimeter: if $A\subset B\subset\R^n$ are two convex bodies, then 
\begin{equation}\label{eq:monotonicity}
\Haus^{n-1}(\de A)\le\Haus^{n-1}(\de B).
\end{equation}

Inequality~\eqref{eq:monotonicity} is well known since the ancient Greek (Archimedes himself took it as a postulate in his work on the sphere and the cylinder, see~\cite{A04}*{p.~36}) and can be proved in many different ways, for example by exploiting either the Cauchy formula for the area surface or the monotonicity property of mixed volumes, \cite{BF87}*{\S7}, by using the Lipschitz property of the projection on a convex closed set, \cite{BFK95}*{Lemma~2.4}, or finally by observing that the perimeter is decreased under intersection with half-spaces, \cite{M12}*{Exercise 15.13}.
Actually, a deep inspection of the proof given in \cite{BFK95} shows that  the convexity of $B$ is not needed.

Anyway, in \cite{LL08},  a quantitative improvement of  formula~\eqref{eq:monotonicity} has been obtained if~$A$ and~$B$ are both convex bodies. 
Moreover, lower bounds for the perimeter deficit 
\begin{equation*}
\delta(B,A):=\Haus^{n-1}(\de B)-\Haus^{n-1}(\de A)
\end{equation*}
with respect to the Hausdorff distance of~$A$ and~$B$ have been established for $n=2$ in~\cites{LL08,CGLP15}, for $n=3$ in~\cite{CGLP16} and finally for all~$n\ge2$ in~\cite{S18}.

In particular,  if $A\subset B$ are two convex bodies in $\R^n$, with~$n\ge2$, then
\begin{equation}\label{eq:monotonicity_quant}
\Haus^{n-1}(\de A)+\frac{\omega_{n-1}r^{n-2}h^2}{r+\sqrt{r^2+h^2}}\le\Haus^{n-1}(\de B),
\end{equation}
where 
$\omega_{n}= \frac{\pi^{n/2}}{\Gamma(\frac{n}{2}+1)}$ denotes the  volume of the  unit ball in $\mathbb R^n$, $h=h(A,B)$ is the Hausdorff distance of $A$ and $B$ and
\begin{equation*}
r=\sqrt[n-1]{\frac{\Haus^{n-1}(B\cap\de H)}{\omega_{n-1}}}, 
\qquad 
H=\set{x\in\R^n : \scalar*{b-a,x-a}\le0},
\end{equation*}
with $a\in A$ and $b\in B$ such that $|a-b|=h(A,B)$, see~\cite{S18}*{Corollary~1.2} and \cref{fig:monotonicity}. 

\begin{figure}
\begin{tikzpicture}
\fill [gray, opacity=0.2](0.4,1.4) -- (0.4,-1.4) -- (1.5,-1.2) -- (1.5,1.6) -- (0.4,1.4);
\fill [gray, opacity=0.2] (0.5,0) ellipse (2 and 1.2);
\fill [red, opacity=0.5] (0,0) circle (1);
\draw (0,0) circle (1);
\draw [dashed] (1,0) -- (2.5,0);
\draw (1.75,0) node [above] {\footnotesize $h$};
\draw [dashed] (-1,0) arc (180:360:1 and 0.25);
\draw [dashed] (-1,0) arc (180:360:1 and -0.25);
\draw (0.5,0) ellipse (2 and 1.2);
\draw [dashed] (-1.5,0) arc (180:360:2 and -0.5);
\draw (-1.5,0) arc (180:360:2 and 0.5);
\draw [dashed] (0.4,1.4) -- (0.4,-1.4) -- (1.5,-1.2) -- (1.5,1.6) -- (0.4,1.4);
\draw [dashed] (1,0) ellipse (.3 and 1.15);
\draw (2,-1.2) node {\footnotesize $\de H$};
\draw (-0.2,0.5) node {\footnotesize $A$};
\draw (2.5,0.8) node {\footnotesize $B$};
\draw (1,0.2) node {\footnotesize $a$};
\draw (2.7,0) node {\footnotesize $b$};
\draw [dashed] (1,0) -- (1.2,.88);
\draw (1.35,.9) node {\footnotesize $r$};
\begin{scope}[shift={(4,-1.25)}]
\fill [gray, opacity=0.2] (3,2.5) arc (-90:90:-0.5 and -1.25) -- (5,1.25) -- (3.1,2.5);
\fill [red, opacity=0.5] (3,2.5) -- (1,2.5) arc (-90:90:-0.5 and -1.25) -- (3,0) arc (-90:90:-0.5 and 1.25);
\draw (1,0) arc (-90:90:-0.5 and 1.25);
\draw [dashed] (1,0) arc (-90:90:0.5 and 1.25);
\draw (1,2.5) -- (3.05,2.5);
\draw (1,0) -- (3,0);
\draw (3,0) arc (-90:90:-0.5 and 1.25);
\draw [dashed] (3,0) arc (-90:90:0.5 and 1.25);
\draw (3.1,0) -- (5,1.25);
\draw (3.1,2.5) -- (5,1.25);
\draw [dashed] (3,1.25) -- (5,1.25);
\draw (4,1.25) node [above] {\footnotesize $h$};
\draw [dashed] (3,1.25) -- (3,2.5);
\draw (3,1.8) node[right] {\footnotesize $r$};
\draw (3,1.25) node [below] {\footnotesize $a$};
\draw (5,1.25) node [below] {\footnotesize $b$};
\draw (1.8,1.8) node {\footnotesize $A$};
\draw (4.5,2.2) node {\footnotesize $B$};
\end{scope}
\end{tikzpicture}
\caption{The setting of the estimate~\eqref{eq:monotonicity_quant} (on the left) with an example of equality (on the right).}
\label{fig:monotonicity}	
\end{figure}
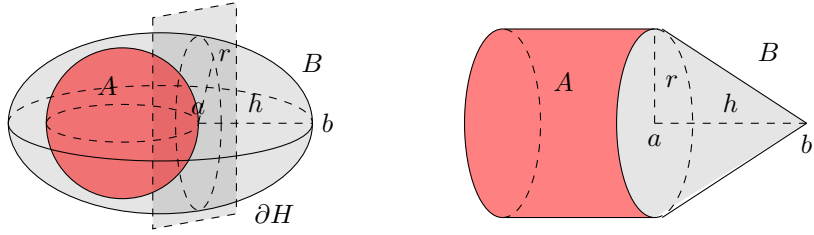

Actually, the main result of~\cite{S18} provides a quantitative lower bound for the more general deficit 
$$\delta_\Phi(B,A):=P_\Phi(B)-P_\Phi(A),$$
where $P_\Phi$ stands for the \emph{anisotropic (Wulff) perimeter} associated to the positively $1$-homogeneous convex function $\Phi\colon\R^n\to[0,+\infty)$.

We conclude this subsection by underlying that  the quantitative estimates of the perimeter deficit $\delta(B,A)$ obtained in \cites{CGLP15,CGLP16,S18} are sharp in the sense that they hold as equalities in some cases, see \cref{fig:monotonicity}.

\subsection{Improvement \texorpdfstring{of~\eqref{eq:leonetti_deficit}}{of the estimate} in the planar case} \label{ss:improvement}
 
Taking advantage of the quantitative estimate~\eqref{eq:monotonicity_quant} in the planar case proved in~\cite{CGLP16}, in the more recent paper~\cite{CGLP19} the authors were able to improve the lower bound~\eqref{eq:leonetti_deficit} for $n=2$ for a class of compact sets $E\subset\R^2$ (see also \cite{Gia}). 
Precisely, if for a bounded closed $\emptyset\neq E\subset\R^2$ one can find $q\in \N_0$, $p\in\N$ and $\alpha\in(0,1)$ such that any decomposition of the form~\eqref{eq:convex_decomp} admits $p$ convex components $E_{i_1},\dots,E_{i_p}$ such that
\begin{equation}
\label{eq:ass_diam_2}
h(E_{i_j},\co(E))
\ge
\alpha\diam(\co(E))
\quad
\text{for all}\ j=1,\dots,p
\end{equation} 
and 
\begin{equation}
\label{eq:numero_q}
q\mathcal{H}^{1}(\partial (\co(E)))-\Haus^{1}(\partial E)<\frac{4\alpha^2p}{1+\sqrt{1+4\alpha^2}}\,\diam(\co(E)),
\end{equation}
then 
\begin{equation}\label{stima:n=2}
k_{\min}(E)
\ge q+1.
\end{equation}
Inequality~\eqref{stima:n=2} is sharp, in the sense that it holds as an equality in some cases. 
Moreover, it improves the previous lower bound~\eqref{eq:leonetti_deficit} in the case $n=2$. 
Indeed, in \cite{CGLP19} the authors exhibited an example for which~\eqref{eq:leonetti_deficit} gives a strict inequality while, on the contrary, \eqref{stima:n=2} yields an equality.

The idea behind the inequality~\eqref{stima:n=2} essentially relies on two ingredients.
 On the one hand, the use of the refined  estimate of the deficit obtained in \cite{CGLP15} in place of the monotonicity property of the perimeter  \eqref{eq:monotonicity}.
On the other hand, the idea of assuming~\eqref{eq:ass_diam_2} for a finite number $p$ of the components, according to the observation that some planar sets $E\subset\R^2$  have some convex components whose Hausdorff distance from the convex hull $\co(E)$ is comparable to the diameter of $\co(E)$ itself, independently of the chosen decomposition. 

By a careful inspection of the proof of \eqref{stima:n=2}, one realizes that
\begin{equation*}
k_{\min}(E)
\ge 
\ceil*{
\frac{
\Haus^1(\de E)+\sum_{j=1}^p\tfrac{4h(E_{i_j},\co(E))^2}{\diam(\co(E))+\sqrt{\diam(\co(E))^2+4h(E_{i_j},\co(E))^2}}
}{
\Haus^1(\de(\co(E)))
}}
\end{equation*}
and since the function
\begin{equation*}
r\mapsto\frac{4h}{r+\sqrt{r^2+4h^2}}
\end{equation*}
is monotone for~$r>0$, the assumption \eqref{eq:ass_diam_2} yields
\begin{equation}
\label{eq:k_min_quantit_2}
k_{\min}(E)
\ge 
\ceil*{
\frac{
\Haus^1(\de E)+\tfrac{4\alpha^2p}{1+\sqrt{1+4\alpha^2}}\,\diam(\co(E))
}{
\Haus^1(\de(\co(E)))
}
},
\end{equation}
which is precisely~\eqref{stima:n=2}, according to the best possible choice of $q\in\N_0$ in~\eqref{eq:numero_q}.

\subsection{Main result} 

The aim of the present paper is to improve the inequality~\eqref{eq:leonetti_deficit} for all $n\ge2$ exploiting the quantitative monotonicity of the perimeter~\eqref{eq:monotonicity_quant} proved in~\cite{S18}, thus generalizing inequality~\eqref{eq:k_min_quantit_2} to higher dimensions.
Before stating our main result, we need to introduce the following notation.

\begin{definition}[Maximal sectional radius]
\label{def:max_sec_radius}
Let $n\ge2$ and let $E\subset\R^n$ be a compact set with non-empty interior.
Given a direction $\nu\in\R^n$, we let
\begin{equation*}
\rho_\nu(E)
=
\sup\set*{\sqrt[n-1]{\frac{\Haus^{n-1}(E\cap(t\nu+\de H_\nu))}{\omega_{n-1}}} : t\in\R}
\end{equation*}
be the \emph{maximal sectional radius of $E$ in the direction $\nu$}, where 
$H_\nu=\set*{x\in\R^n : \scalar*{x,\nu}\le0}$.
Note that, naturally, $\rho_{-\nu}(E)=\rho_{\nu}(E)$ for all $\nu\in\R^n$.
\end{definition}

With the above definition in force, our main result reads as follows.

\begin{theorem}\label{res:main}
Let $n\ge2$ and let $E\subset\R^n$ be a compact set with non-empty interior. 
Assume that there exist $p\in\N$, $\alpha\in(0,1)$ and $\beta\in[0,1]$ with the following properties. 
For every family $E_1,\dots,E_k$, with $k\in\N$, of convex bodies with non-empty interior such that $E=\bigcup_{i=1}^k E_i$, we can find a subfamily of~$p$ convex bodies $E_{i_1},\dots,E_{i_p}$ and a family of corresponding~$p$ closed half-spaces such that $E_{i_j}\subset H_{i_j}$, 
\begin{equation}
\label{eq:ass_diam}
h(\co(E),\co(E)\cap H_{i_j})
\ge
\alpha\diam(\co(E))
\end{equation}
and 
\begin{equation}
\label{eq:ass_section}
\Haus^{n-1}(\co(E)\cap\de H_{i_j})\ge
\beta\omega_{n-1}\rho_{\nu_{i_j}}(\co(E))^{n-1}
\end{equation}
for all $j=1,\dots,p$, 
where $\nu_{i_j}=a_{i_j}-b_{i_j}$, with $a_{i_j}\in\co(E)\cap H_{i_j}$ and $b_{i_j}\in\co(E)$ such that $h(\co(E)\cap H_{i_j},\co(E))
=|a_{i_j}-b_{i_j}|$ and $H_{i_j}=\set*{x\in\R^n : \scalar*{b_{i_j}-a_{i_j},x-a_{i_j}}\le0}$.
Then
\begin{equation}
\label{eq:main_k_est}
k_{\min}(E)
\ge
\ceil*{ 
\frac{
\Haus^{n-1}(\de E)
+
\omega_{n-1}\alpha^2\beta^{\frac{n-2}{n-1}}
\displaystyle\sum_{j=1}^p
\tfrac{\rho_{\nu_{i_j}}(\co(E))^{n-2}\diam(\co(E))^2}{\rho_{\nu_{i_j}}(\co(E))+\sqrt{\rho_{\nu_{i_j}}(\co(E))^2+\alpha^2\diam(\co(E))^2}}
}
{\Haus^{n-1}(\de(\co(E)))}
}.
\end{equation}
\end{theorem}

\subsection{Comments}

First of all, let us remark that inequality~\eqref{eq:main_k_est} improves the previous lower bound~\eqref{eq:leonetti_deficit}.
Indeed, inequality~\eqref{eq:main_k_est} clearly reduces to the lower bound~\eqref{eq:leonetti_deficit} as soon as one drops the additional assumptions on each of all possible decompositions of the form~\eqref{eq:convex_decomp}.
Moreover, inequality~\eqref{eq:main_k_est} holds as an equality in some cases for which~\eqref{eq:leonetti_deficit} gives  a strict inequality only.
We will give some explicit examples in \cref{sec:examples} below.

Concerning the statement of \cref{res:main}, it is worth noting that the assumption~\eqref{eq:ass_diam} corresponds to~\eqref{eq:ass_diam_2}, while the additional assumption~\eqref{eq:ass_section} comes into play for $n\ge3$ only.

In fact, if we take $n=2$ in \cref{res:main}, then the inequality~\eqref{eq:main_k_est} becomes
\begin{equation}
\label{eq:main_k_est_2}
k_{\min}(E)
\ge
\ceil*{ 
\frac{
\Haus^{1}(\de E)
+
2\alpha^2
\displaystyle\sum_{j=1}^p
\tfrac{
\diam(\co(E))^2
}
{
\rho_{\nu_{i_j}}(\co(E))+\sqrt{\rho_{\nu_{i_j}}(\co(E))^2
+
\alpha^2\diam(\co(E))^2}}
}
{\Haus^{1}(\de(\co(E)))}
}
\end{equation} 
(as it is customary, we use the convention $0^0=1$) and the parameter $\beta\in[0,1]$ provided by~\eqref{eq:ass_section} plays no role in the final estimate~\eqref{eq:main_k_est_2}.
Consequently, the additional assumption in~\eqref{eq:ass_section} can be dropped and one just needs to choose the closed half-plane $H_{i_j}\subset\R^2$ in such a way that
\begin{equation*}
h(\co(E)\cap H_{i_j},\co(E))
=
h(E_{i_j},\co(E))
\quad
\text{for all}\ j=1,\dots,p,
\end{equation*}
which is always possible by the definition of the Hausdorff distance and the convexity of each component~$E_{i_j}$.

Concerning the higher dimensional case $n\ge3$, a control like the one in~\eqref{eq:ass_section} seems reasonable to be assumed.
Indeed, as one may realize by looking at the inequality~\eqref{eq:leonetti_deficit}, the set $E\subset\R^n$ may have a convex component very lengthened in one specific direction $\nu\in\mathbb S^{n-1}$ which does not give a substantial contribution to the total perimeter of $E$ but, nevertheless,  that strongly affects the total perimeter of the convex hull $\co(E)$.

In addition, we observe that the effectiveness of the lower bound~\eqref{eq:leonetti_deficit} drastically changes when passing from the planar case $n=2$ to the non-planar case $n\ge3$.
Indeed, if $E\subset\R^2$ is a non-convex connected compact set admitting at least one decomposition like~\eqref{eq:convex_decomp}, then
\begin{equation*}
\Haus^1(\de(\co(E)))
<
\Haus^1(\de E),
\end{equation*} 
correctly implying that $k_{\min}(E)\ge2$.
As a matter of fact, in the planar case $n=2$, the examples given in~\cite{CGLP19} provide the precise value of $k_{\min}(E)$ for $q\ge2$, since if $q=1$ both inequalities~\eqref{eq:leonetti_deficit} and~\eqref{stima:n=2} allow to conclude that $k_{\min}(E)\ge2$ only.
However, as we are going to show with some examples in \cref{sec:examples} below, there are non-convex connected compact sets $E\subset\R^n$, with $n\ge3$, such that
\begin{equation*}
\Haus^{n-1}(\de(\co(E)))
\ge
\Haus^{n-1}(\de E),
\end{equation*}  
so that~\eqref{eq:leonetti_deficit} only implies that $k_{\min}(E)\ge1$.
Nevertheless, the inequality~\eqref{eq:main_k_est} given by \cref{res:main} allows us to recover the correct value of $k_{\min}(E)$ in these examples.

Moreover,  let us observe that, in the planar case $n=2$,  one can trivially bound 
\begin{equation}
\label{eq:max_radius_diam_est}
\rho_\nu(\co(E))\le\frac{\diam(\co(E))}2
\quad
\text{for all}\ 
\nu\in\mathbb S^1,	
\end{equation}
 so that inequality~\eqref{eq:main_k_est_2} gives back 
\begin{align*}
k_{\min}(E)
&\ge 
\ceil*{
\frac{
\Haus^{1}(\de E)
+
\tfrac{
2p\alpha^2
\diam(\co(E))^2
}
{
\frac{\diam(\co(E))}{2}+\sqrt{\frac{\diam(\co(E))^2}{4}
+
\alpha^2\diam(\co(E))^2}}
}
{\Haus^{1}(\de(\co(E)))}
}
\\
&=
\ceil*{
\frac{
\Haus^{1}(\de E)
+
\tfrac{
4\alpha^2 p
}{
1+\sqrt{1+4\alpha^2}
}\diam(\co(E))
}{
\Haus^{1}(\de(\co(E)))
}
},
\end{align*}
that is the estimate in~\eqref{eq:k_min_quantit_2}.
Actually, because of the fact that the upper bound~\eqref{eq:max_radius_diam_est} can be too rough in general, the inequality~\eqref{eq:main_k_est_2} given by our \cref{res:main} is more precise than the one in~\eqref{eq:k_min_quantit_2}, as we are going to show in \cref{exa:impro_2} below.

Last but not least, we remark that both the lower bounds provided by the estimates~\eqref{eq:leonetti_deficit} and~\eqref{eq:main_k_est} are not stable under small modifications of the compact set $E\subset\R^n$, $n\ge2$. In fact, the value of $k_{\min}(E)$ may be changed without substantially altering neither the perimeters of $E$ and of its convex hull $\co(E)$, nor all the other geometrical quantities involved in~\eqref{eq:main_k_est}, for example by gluing some additional tiny convex components to the original set $E$.

\subsection{Organization of the paper}
The rest of the paper is organized as follows.

In \cref{sec:proof} we  detail the proof of our main result \cref{res:main}. 
Our approach essentially follows the strategy of~\cite{CGLP19}, up to some minor modifications needed in order to exploit the quantitative estimate~\eqref{eq:monotonicity_quant} in conjunction with the notion of maximal radius introduced in \cref{def:max_sec_radius}. 

In \cref{sec:examples} we provide some examples proving the effectiveness of our main result with respect to either the general inequality~\eqref{eq:leonetti_deficit} or its improvement~\eqref{eq:k_min_quantit_2} in the planar case, as already observed, due to the fact that $\rho_\nu(\co(E))\le\frac{\diam(\co(E))}2$
for all $\nu\in\mathbb S^1$.

\section{Proof of \texorpdfstring{\cref{res:main}}{the main result}}
\label{sec:proof}

We recall that, if $A\subset B$ are two compact sets in $\R^n$, with $n\ge2$, then the Hausdorff distance $h(A,B)$ between $A$ and $B$ can be written as
\begin{equation*}
h(A,B)
=
\max_{b\in B}\dist(A,b)
=
\max_{b\in B}\min_{a\in A} |a-b|.
\end{equation*}
As above, given $\emptyset\neq A\subset B$ two convex bodies in $\R^n$, with $n\ge2$, we denote by
\begin{equation*}
\delta(B,A)
:=
\Haus^{n-1}(\de B)
-
\Haus^{n-1}(\de A)\ge0
\end{equation*}
the perimeter deficit between $A$ and $B$.

\begin{proof}[Proof of \cref{res:main}]
Since $E$ is compact, its convex hull $\co(E)$ is compact	too, see~\cite{G07}*{Corollary~3.1} for example.
As a consequence, $\Haus^{n-1}(\de(\co(E)))<+\infty$.
Arguing as in~\cite{CGLP19}*{Section~2}, we can estimate
\begin{align*}
\Haus^{n-1}(\de E)
&\le
\Haus^{n-1}\left(\bigcup_{i=1}^k \de E_i\right)
\le
\sum_{i=1}^k\Haus^{n-1}(\de E_i)
\\
&=\
\sum_{j=1}^p\Haus^{n-1}(\de E_{i_j})
+
\sum_{j=p+1}^k\Haus^{n-1}(\de E_{i_j})
\\
&\le
\sum_{j=1}^p\left(\Haus^{n-1}(\de(\co(E)))-\delta(\co(E),E_{i_j})\right)
+
\sum_{j=p+1}^k\Haus^{n-1}(\de(\co(E)))
\\
&=\ 
k\Haus^{n-1}(\de(\co(E)))-\sum_{j=1}^p\delta(\co(E),E_{i_j}),
\end{align*}
so that
\begin{equation*}
\ceil*{
\frac{
\Haus^{n-1}(\de E)
+
\displaystyle
\sum_{j=1}^p\delta(\co(E),E_{i_j})
}
{\Haus^{n-1}(\de(\co(E)))}
}
\le 
k.
\end{equation*}
Now, since $E_{i_j}\subset H_{i_j}$, we observe that 
\begin{equation}
\label{eq:stima_referee}
\begin{split}
\delta(\co(E),E_{i_j})
&=
\Haus^{n-1}(\de(\co(E)))
-
\Haus^{n-1}(\de E_{i_j})\\
&=
\left(\Haus^{n-1}(\de(\co(E)))
-
\Haus^{n-1}(\de(\co(E)\cap H_{i_j}))
\right)
\\
&\quad+
\left(
\Haus^{n-1}(\de(\co(E)\cap H_{i_j}))
-
\Haus^{n-1}(\de E_{i_j})
\right)
\\
&=
\delta(\co(E)\cap H_{i_j},E_{i_j})
+
\delta(\co(E),\co(E)\cap H_{i_j})
\\
&\ge
\delta(\co(E),\co(E)\cap H_{i_j})
\end{split}
\end{equation}
for all $j=1,\dots,p$.
Since $h(\co(E)\cap H_{i_j},\co(E))
=|a_{i_j}-b_{i_j}|$  with $a_{i_j}\in\co(E)\cap H_{i_j}$ and $b_{i_j}\in\co(E)$ such that 
$$H_{i_j}=\set*{x\in\R^n : \scalar*{b_{i_j}-a_{i_j},x-a_{i_j}}\le0},$$ 
we can thus apply~\eqref{eq:monotonicity_quant} to each couple of convex bodies $\co(E)$ and $\co(E)\cap H_{i_j}$, with $j=1,\dots,p$, and get
\begin{equation}
\label{eq:deficit_conv_comp}
\delta(\co(E),\co(E)\cap H_{i_j})
\ge
\frac{\omega_{n-1}r^{n-2}_{i_j}h_{i_j}^2}{r_{i_j}+\sqrt{r_{i_j}^2+h_{i_j}^2}},
\end{equation}
where
\begin{equation*}
h_{i_j}=h(\co(E),\co(E)\cap H_{i_j}),
\qquad
r_{i_j}
=
\sqrt[n-1]{\frac{\Haus^{n-1}(\co(E)\cap\de H_{i_j})}{\omega_{n-1}}}.
\end{equation*}
By~\eqref{eq:ass_section}, we clearly have
\begin{equation}
\label{eq:max_sect_rad_est}
\beta^{\frac1{n-1}}\rho_{\nu_{i_j}}(\co(E))\le r_{i_j}\le\rho_{\nu_{i_j}}(\co(E))
\end{equation}
for all $j=1,\dots,p$. 
Inserting~\eqref{eq:max_sect_rad_est} into~\eqref{eq:deficit_conv_comp}, we immediately obtain that
\begin{align*}
\delta(\co(E),\co(E)\cap H_{i_j})
\ge
\frac{\omega_{n-1}\beta^{\frac{n-2}{n-1}}\rho_{\nu_{i_j}}(\co(E))^{n-2}h_{i_j}^2}{\rho_{\nu_{i_j}}(\co(E))+\sqrt{\rho_{\nu_{i_j}}(\co(E))^2+h_{i_j}^2}}
\end{align*}
for all $j=1,\dots,p$.
Now, for any given $c>0$, the function
\begin{equation*}
s\mapsto\frac{s^2}{c+\sqrt{c+s^2}}
\end{equation*}
is strictly increasing for $s>0$. 
Since $h_{i_j}\ge\alpha\diam(\co(E))$ for all $j=1,\dots,p$ by~\eqref{eq:ass_diam}, thanks to~\eqref{eq:stima_referee} we can finally estimate
\begin{equation*}
\delta(\co(E),E_{i_j})
\ge
\frac{\omega_{n-1}\alpha^2\beta^{\frac{n-2}{n-1}}\rho_{\nu_{i_j}}(\co(E))^{n-2}\diam(\co(E))^2}{\rho_{\nu_{i_j}}(\co(E))+\sqrt{\rho_{\nu_{i_j}}(\co(E))^2+\alpha^2\diam(\co(E))^2}}
\end{equation*}
for all $j=1,\dots,p$.
In conclusion, we get
\begin{align*}
k
&\ge
\ceil*{
\frac{
\Haus^{n-1}(\de E)
+
\displaystyle
\sum_{j=1}^p\delta(E_{i_j},\co(E))
}
{\Haus^{n-1}(\de(\co(E)))}
}
\\
&\ge
\ceil*{
\frac{
\Haus^{n-1}(\de E)
+
\omega_{n-1}\alpha^2\beta^{\frac{n-2}{n-1}}
\displaystyle\sum_{j=1}^p
\tfrac{\rho_{\nu_{i_j}}(\co(E))^{n-2}\diam(\co(E))^2}{\rho_{\nu_{i_j}}(\co(E))+\sqrt{\rho_{\nu_{i_j}}(\co(E))^2+\alpha^2\diam(\co(E))^2}}
}
{\Haus^{n-1}(\de(\co(E)))}
}
\end{align*}
proving~\eqref{eq:main_k_est}.
The proof is thus complete.
\end{proof}

\section{Examples}
\label{sec:examples}

We dedicate the remaining part of the paper to give some explicit examples of compact sets $E\subset\R^n$, $n\ge2$, for which our main result applies.
In each example, we will identify a point $P\in\de E$ and one convex component $E_j$ of $E$ containing~$P$ and we will make a precise choice of parameters in order to satisfy the hypotheses of \cref{res:main}.

\subsection{An example in \texorpdfstring{$\R^2$}{Rˆ2}}

We begin with the following example in $\R^2$ showing that our \cref{res:main} in the planar formulation~\eqref{eq:main_k_est_2}, at least in some cases, provides a strictly better estimate than the one in~\eqref{eq:k_min_quantit_2} previously established in~\cite{CGLP19}.
This example is based on the set $C\subset\R^2$ shown in \cref{fig:C_shape}, which was already considered in~\cite{LL08}*{Example~2.1} and in~\cite{CGLP19}*{Example~3.1}.
The set $C$ depends on two parameters $l>h>0$. In~~\cite{CGLP19}*{Example~3.1}, to make the construction work, it was necessary to assume that $h\in(0,\eps)$ for some $\eps\in(0,l)$ sufficiently small. In our situation, thanks to the refined inequality~\eqref{eq:k_min_quantit_2}, our choice of the parameter $h$ is less restrictive, i.e., we are going to choose $h\in(0,\bar\eps)$ for some $\bar\eps\in(\eps,l)$. As matter of fact, when $h\in(\eps,\bar\eps)$, our inequality~\eqref{eq:k_min_quantit_2} gives the correct value $k_{\min}(C)=3$, while inequality~\eqref{stima:n=2} gives the lower bound $k_{\min}(C)\ge2$ only.

\begin{figure}[H]
\begin{tikzpicture}[scale=0.65]
\fill [gray, opacity=0.2] (0,0) -- (0,3) -- (10,3) -- (10,2) -- (1,2) -- (1,1) -- (10,1) -- (10,0) -- (0,0);
\draw (0,0) -- (0,3) -- (10,3) -- (10,2) -- (1,2) -- (1,1) -- (10,1) -- (10,0) -- (0,0); 
\draw (1,1.5) node {.};
\draw (1,1.5) node [right] {\footnotesize $P$};
\draw [decorate,decoration={brace,amplitude=4pt},xshift=-2pt,yshift=0pt] (0,0) -- (0,3) node [black,midway,xshift=-10pt] {\footnotesize $3h$};
\draw [decorate,decoration={brace,amplitude=2pt,mirror},xshift=0pt,yshift=-2pt] (0,0) -- (1,0) node [black,midway,yshift=-8pt] {\footnotesize $h$};
\draw [decorate,decoration={brace,amplitude=2pt,mirror},xshift=0pt,xshift=2pt] (10,0) -- (10,1) node [black,midway,xshift=6pt] {\footnotesize $h$};
\draw [decorate,decoration={brace,amplitude=2pt,mirror},xshift=0pt,xshift=2pt] (10,2) -- (10,3) node [black,midway,xshift=6pt] {\footnotesize $h$};
\draw [decorate,decoration={brace,amplitude=4pt},xshift=0pt,yshift=3pt] (0,3) -- (10,3) node [black,midway,yshift=10pt] {\footnotesize $l$};
\begin{scope}[shift={(12,0)}]
\fill [gray, opacity=0.2] (0,0) -- (0,3) -- (10,3) -- (10,0) -- (0,0);
\draw (0,0) -- (0,3) -- (10,3) -- (10,0) -- (0,0);
\draw [dashed] (1,-0.75) -- (1,3.75);
\draw (1,3.5) node [right] {\footnotesize $\de H_j$};
\draw [<-] (0,3.45) -- (0.75,3.45);
\draw (0.5,3.4) node [above] {\footnotesize $\nu_j$};
\draw [dashed, very thin] (1,1) -- (10,1);
\draw [dashed, very thin] (1,2) -- (10,2);
\end{scope}
\end{tikzpicture}
\caption{The set $C\subset\R^2$ (on the left) and its convex hull (on the right).}
\label{fig:C_shape}
\end{figure}
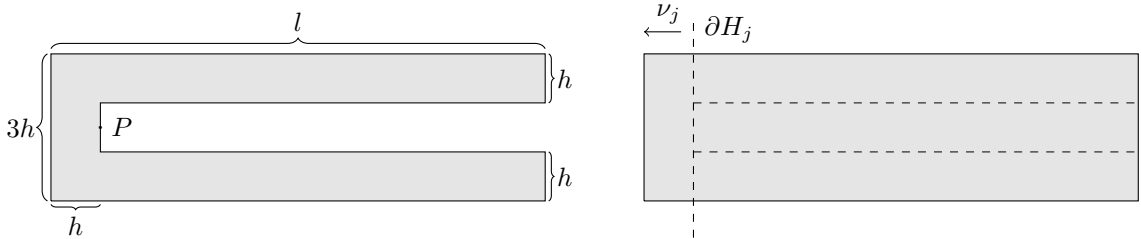

\begin{example}[The set $C\subset\R^2$]
\label{exa:impro_2}
Let $l>h>0$ and consider the set $C\subset\R^2$ in \cref{fig:C_shape}. 
We can compute
\begin{align*}
\Haus^1(\de C)
=
4l+4h,
\quad
\Haus^1(\de(\co(C)))
=
2l+6h,
\quad
\diam(\co(C))
=
\sqrt{l^2+9h^2}.
\end{align*}
Since $C$ is not convex, we must have that $k_{\min}(C)\ge2$.   
After all, it is evident that $k_{\min}(C)=3$.
Our argument will give such right value for a larger class of parameters $l>h>0$ than the one provided in~\cite{CGLP19}*{Example~3.1}.
First of all, notice that we do not deduce any further information from the result in~\cite{LL08}.
Indeed, inequality~\eqref{eq:leonetti_deficit} only yields
\begin{equation*}
k_{\min}(C)
\ge
\ceil*{
\frac{
\Haus^1(\de C)
}{
\Haus^1(\de(\co(C)))
}
}=2,
\end{equation*}
since an elementary computation shows that
\begin{equation*}
\frac{
\Haus^1(\de C)
}{
\Haus^1(\de(\co(C)))
}
=
\frac{
2l+2h
}{
l+3h
}
\in(1,2)
\end{equation*}
whenever $l>h>0$.
We now consider the point $P\in\de C$ as shown in \cref{fig:C_shape}. For every decomposition of $C$ into convex bodies, there exists a convex body $E_j$ containing~$P$. 
Since $E_j$ is convex and contained in $C$, we must have that $E_j\subset H_j$, where $H_j$ is the half-space such that $\de H_j$ contains the face of $C$ to which the point $P$ belongs, see \cref{fig:C_shape}.
Consequently, we must have
\begin{align*}
h(\co(C)\cap H_j,\co(C))=l-h,
\quad
\Haus^1(\co(C)\cap\de H_j)=3h,
\quad
\rho_{\nu_j}(\co(C))=\frac{3h}2,	
\end{align*}
where $\nu_j\in\mathbb S^1$ is the inner unit normal of the half-space $H_j$ as in \cref{fig:C_shape}.
Now let $l>0$ be fixed.
In \cite{CGLP19}, it has been shown that, for any $\alpha\in (0,1)$, $p=1$ and $h\ll l$, one has
\begin{equation*}
\ceil*{
\frac{
\Haus^1(\de C)+\tfrac{4\alpha^2}{1+\sqrt{1+4\alpha^2}}\,\diam(\co(C))
}{
\Haus^1(\de(\co(C)))
}
}
=3.
\end{equation*}
We now apply inequality~\eqref{eq:k_min_quantit_2} and  \cref{res:main} with 
\begin{equation*}
p=1,
\quad
\alpha=\frac{l-h}{\sqrt{l^2+9h^2}},
\quad
\beta=0.
\end{equation*}
We claim that we can choose $h\in(0,l)$ such that
\begin{equation*}
\ceil*{
\frac{
\Haus^1(\de C)+\tfrac{4\alpha^2}{1+\sqrt{1+4\alpha^2}}\,\diam(\co(C))
}{
\Haus^1(\de(\co(C)))
}
}
=2
\end{equation*}
and 
\begin{equation*}
\ceil*{ 
\frac{
\Haus^{1}(\de C)
+
\tfrac{
2\alpha^2
\diam(\co(C))^2
}
{
\rho_{\nu_j}(\co(C))+\sqrt{\rho_{\nu_j}(\co(C))^2
+
\alpha^2\diam(\co(C))^2}}
}
{\Haus^{1}(\de(\co(C)))}
}
=3.
\end{equation*}
In order to have both the claimed inequalities, it is sufficient to find $h\in (0,l)$ such that
\begin{equation*}
\frac{
\Haus^1(\de C)+\tfrac{4\alpha^2}{1+\sqrt{1+4\alpha^2}}\,\diam(\co(C))
}{
\Haus^1(\de(\co(C)))
}\le 2<
\frac{
\Haus^{1}(\de C)
+
\tfrac{
2\alpha^2
\diam(\co(C))^2
}
{
\rho_{\nu_j}(\co(C))+\sqrt{\rho_{\nu_j}(\co(C))^2
+
\alpha^2\diam(\co(C))^2}}
}
{\Haus^{1}(\de(\co(C)))},
\end{equation*}
that is,
\begin{equation*}
\frac{
2(l+h)+\tfrac{2\alpha^2}{1+\sqrt{1+4\alpha^2}}\,\sqrt{l^2+9h^2}
}{
l+3h
}\le 2<
\frac{
2(l+h)
+
\tfrac{
2\alpha^2
(l^2+9h^2)
}
{
3h+\sqrt{9h^2
+
4\alpha^2(l^2+9h^2)}
}}
{l+3h}.
\end{equation*}
Up to some elementary algebraic computations, we need to find $h\in(0,l)$ such that
\begin{equation*}
\frac{
(l-h)^2
}{
3h+\sqrt{9h^2+4(l-h)^2}
}
>2h\ge
\frac{
(l-h)^2
}{
\sqrt{l^2+9h^2}
+
\sqrt{l^2+9h^2+4(l-h)^2}
}.
\end{equation*}
If we let $h=tl$ for $t\in(0,1)$, then we just need to solve
\begin{equation*}
\begin{cases}
1-5t^2-2t-2t\sqrt{9t^2+4(1-t)^2}>0\\[2mm]
2t\sqrt{1+9t^2}+2t\sqrt{1+9t^2+4(1-t)^2}-1-t^2+2t\ge0
\end{cases}
\end{equation*}
and we let the reader check that the above system of inequalities admits solutions. 
\end{example}

\subsection{Some examples in \texorpdfstring{$\R^3$}{Rˆ3}}

We now give some examples in $\R^3$ showing that for $n=3$ our \cref{res:main} provides an improvement of the inequality~\eqref{eq:leonetti_deficit} established in~\cite{LL08}.

\begin{figure}[H]
\begin{tikzpicture}[scale=0.65]
\fill [gray, opacity=0.2] (0,0) -- (0,2) -- (0.5,2.5) -- (10.5,2.5) -- (10.5,1.5) -- (10,1) -- (1.5,1) -- (1.5,0.5) -- (1,0);

\draw (0,0) -- (0,2) -- (10,2) -- (10,1) -- (1,1) -- (1,0) -- (0,0); 
\draw (0,2) -- (0.5,2.5) -- (10.5,2.5);
\draw (10,2) -- (10.5,2.5);	
\draw (10,1) -- (10.5,1.5);
\draw (10.5,1.5) -- (10.5,2.5);
\draw (1,0) -- (1.5,0.5);
\draw (1.5,0.5) -- (1.5,1);
\draw [dashed] (1,1) -- (1.5,1.5);
\draw [dashed] (1.5,1.5) -- (10.5,1.5);
\draw [dashed] (1.5,1) -- (1.5,1.5);
\draw [dashed] (0,0) -- (0.5,0.5);
\draw [dashed] (0.5,0.5) -- (1.5,0.5);
\draw [dashed] (0.5,0.5) -- (0.5,2.5);
\draw [decorate,decoration={brace,amplitude=3pt},xshift=-2pt,yshift=0pt] (0,0) -- (0,2) node [black,midway,xshift=-10pt] {\footnotesize $2h$};
\draw [decorate,decoration={brace,amplitude=2pt,mirror},xshift=0pt,yshift=-2pt] (0,0) -- (1,0) node [black,midway,yshift=-8pt] {\footnotesize $h$};
\draw [decorate,decoration={brace,amplitude=2pt,mirror},xshift=0pt,xshift=2pt,yshift=-1pt] (1,0) -- (1.5,0.5) node [black,midway,xshift=5pt,yshift=-5pt] {\footnotesize $h$};
\draw [decorate,decoration={brace,amplitude=2pt,mirror},xshift=0pt,xshift=2pt,yshift=-1pt] (10,1) -- (10.5,1.5) node [black,midway,xshift=5pt,yshift=-5pt] {\footnotesize $h$};
\draw [decorate,decoration={brace,amplitude=2pt,mirror},xshift=0pt,xshift=2pt] (10.5,1.5) -- (10.5,2.5) node [black,midway,xshift=6pt] {\footnotesize $h$};
\draw [decorate,decoration={brace,amplitude=4pt},xshift=0pt,yshift=17pt] (0.5,2) -- (10.5,2) node [black,midway,yshift=10pt] {\footnotesize $l$};
\draw (1.25,0.75) node {.};	
\draw (1.25,0.85) node [below] {\footnotesize $P$};
\begin{scope}[shift={(12,0)}]
\fill [gray, opacity=0.2] (1,-0.75) -- (1,3) -- (1.5,3.55) -- (1.5,-0.25) -- (1,-0.75);
\fill [gray, opacity=0.2] (0,0) -- (0,2) -- (0.5,2.5) -- (10.5,2.5) -- (10.5,1.5) -- (10,1) -- (1,0);
\draw (0,0) -- (0,2) -- (10,2) -- (10,1)  -- (1,0) -- (0,0); 
\draw (0,2) -- (0.5,2.5) -- (10.5,2.5);
\draw (10,2) -- (10.5,2.5);	
\draw (10,1) -- (10.5,1.5);
\draw (10.5,1.5) -- (10.5,2.5);
\draw [dashed] (1.5,0.5) -- (10.5,1.5);
\draw [dashed] (0,0) -- (0.5,0.5);
\draw [dashed] (1,0) -- (1.5,0.5);
\draw [dashed] (1,2) -- (1.5,2.5);
\draw [dashed] (0.5,0.5) -- (1.5,0.5);
\draw [dashed] (0.5,0.5) -- (0.5,2.5);
\draw [<-] (0,3) -- (0.85,3);
\draw (0.5,3.4) node {\footnotesize $\nu_j$};
\draw [dashed] (1,-0.75) -- (1,3) -- (1.5,3.55) -- (1.5,-0.25) -- (1,-0.75);
\draw (2.25,2.5) node [right,above] {\footnotesize $\de H_j$};
\end{scope}
\end{tikzpicture}
\caption{The set $L\subset\R^3$ (on the left) and its convex hull (on the right).}
\label{fig:L_shape}
\end{figure}
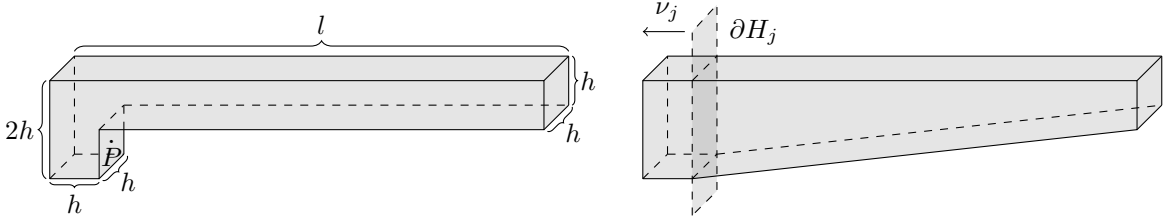

\begin{example}[The set $L\subset\R^3$]
\label{exa:L_shape}
Let $l>h>0$ and consider the set $L\subset\R^3$ in \cref{fig:L_shape}.
We can compute
\begin{align*}
\Haus^2(\de L)
&=
4hl+6h^2,
\\
\Haus^2(\de(\co(L)))
&=
4hl+5h^2+h\sqrt{(l-h)^2+h^2},
\\
\diam(\co(L))
&=
\sqrt{l^2+5h^2}.
\end{align*}
Since $L$ is not convex, we must have that $k_{\min}(L)\ge2$, and a simple geometric argument allows to conclude that $k_{\min}(L)=2$.
From~\eqref{eq:leonetti_deficit} we deduce that
\begin{equation*}
k_{\min}(L)\ge
\ceil*{
\frac{\Haus^{2}(\de L)}{\Haus^{2}(\de(\co(L)))}
}=1,
\end{equation*}
since an elementary computation shows that
\begin{equation*}
\frac{\Haus^{2}(\de L)}{\Haus^{2}(\de(\co(L)))}
=
\frac{
4l+6h
}{
4l+5h+\sqrt{(l-h)^2+h^2}
}
\in(0,1)
\end{equation*}
whenever $l>h>0$.
We now consider the point $P\in\de L$ as shown in \cref{fig:L_shape}. 
For every decomposition of $L$ into convex bodies, there exists a convex body $E_j$ containing~$P$.
Since $E_j$ is convex and contained in $L$, we must have that $E_j\subset H_j$, where $H_j$ is the half-space such that $\de H_j$ contains the face of $L$ to which the point $P$ belongs, see \cref{fig:L_shape}.
Consequently, we must have
\begin{align*}
h(\co(L)\cap H_j,\co(L))=l-h,
\quad
\Haus^2(\co(L)\cap\de H_j)=2h^2,
\quad
\rho_{\nu_j}(\co(L))=\sqrt{\frac{2h^2}{\pi}},	
\end{align*}
where $\nu_j\in\mathbb S^2$ is the inner unit normal of the half-space $H_j$ as in \cref{fig:L_shape}.
We now let $l>0$ be fixed.
We apply \cref{res:main} with
\begin{align*}
p=1,
\quad
\alpha=\frac{l-h}{\sqrt{l^2+5h^2}},
\quad
\beta=1.
\end{align*}
Provided that we choose $h\in(0,l)$ sufficiently small, we conclude that
\begin{align*}
k_{\min}(L)
&\ge
\ceil*{
\frac{
4hl+6h^2
+
\pi\left(\tfrac{l-h}{\sqrt{l^2+5h^2}}\right)^2
\tfrac{
\sqrt{\tfrac{2h^2}{\pi}}
\,
\left(\displaystyle\sqrt{l^2+5h^2}\right)^2
}
{
\sqrt{\tfrac{2h^2}{\pi}}
+
\sqrt{\tfrac{2h^2}{\pi}
+
\left(\tfrac{l-h}{\sqrt{l^2+5h^2}}\right)^2
\,
\left(\displaystyle\sqrt{l^2+5h^2}\right)^2}}
}
{4hl+5h^2+h\sqrt{(l-h)^2+h^2}}
}
\\
&=
\ceil*{
\frac{
4l+6h
+
\tfrac{\sqrt{2\pi}\,(l-h)^2}
{
\sqrt{\tfrac{2h^2}{\pi}}
+
\sqrt{\tfrac{2h^2}{\pi}
+
(l-h)^2}
}
}
{4l+5h+\sqrt{(l-h)^2+h^2}}
}
=2,
\end{align*} 
since
\begin{equation*}
\lim_{h\to0^+}
\frac{
4l+6h
+
\tfrac{\sqrt{2\pi}\,(l-h)^2}
{
\sqrt{\tfrac{2h^2}{\pi}}
+
\sqrt{\tfrac{2h^2}{\pi}
+
(l-h)^2}
}
}
{4l+5h+\sqrt{(l-h)^2+h^2}}
=
\frac{4+\sqrt{2\pi}}{5}
\in(1,2).
\end{equation*}
\end{example}

\begin{figure}[H]
\begin{tikzpicture}[scale=0.65]
\fill [gray, opacity=0.2] (0,0) -- (0,3) -- (1.25,4.75) -- (11.25,4.75) -- (11.25,2.75) -- (10,1) -- (1,0);
\fill[gray, opacity=0.2] (1,1) -- (1,2) -- (9,2); 
\draw (0,0) -- (0,3) -- (10,3) -- (10,1) -- (1,0) -- (0,0);
\draw (1,1) -- (1,2) -- (9,2) -- (1,1);
\draw (0,3) -- (1.25,4.75);
\draw (10,3) -- (11.25,4.75);	
\draw (10,1) -- (11.25,2.75);
\draw [dashed] (1.25,1.75) -- (1.25,4.75);
\draw (1.25,4.75) -- (11.25,4.75) -- (11.25,2.75);
\draw [dashed] (0,0) -- (1.25,1.75);
\draw [dashed] (1,0) -- (2.25,1.75);
\draw [dashed] (1,2) -- (2.25,3.75) -- (10.25,3.75) -- (9,2);
\draw (1,1) -- (1.7,2);
\draw[dashed] (1.7,2) -- (2.25,2.75);
\draw[dashed] (2.25,2.75) -- (2.25,3.75);
\draw[dashed] (2.25,2.75) -- (10.25,3.75);
\draw[dashed] (1.25,1.75) -- (2.25,1.75) -- (11.25,2.75);
\draw [decorate,decoration={brace,amplitude=3pt},xshift=-2pt,yshift=0pt] (0,0) -- (0,3) node [black,midway,xshift=-10pt] {\footnotesize $3h$};
\draw [decorate,decoration={brace,amplitude=2pt},xshift=-1pt,yshift=1pt] (0,3) -- (1.25,4.75) node [black,midway,xshift=-8pt,yshift=8pt] {\footnotesize $4h$};
\draw [decorate,decoration={brace,amplitude=2pt,mirror},xshift=0pt,yshift=-2pt] (0,0) -- (1,0) node [black,midway,xshift=0pt,yshift=-8pt] {\footnotesize $h$};
\draw [decorate,decoration={brace,amplitude=5pt},xshift=0pt,yshift=2pt] (1.25,4.75) -- (11.25,4.75) node [black,midway,xshift=0pt,yshift=12pt] {\footnotesize $l$};
\draw [decorate,decoration={brace,amplitude=3pt},xshift=2pt,yshift=0pt] (11.25,4.75) -- (11.25,2.75) node [black,midway,xshift=10pt,yshift=0pt] {\footnotesize $2h$};
\draw [decorate,decoration={brace,amplitude=2pt},xshift=-1pt,yshift=0pt] (1,1) -- (1,2) node [black,midway,xshift=-5pt,yshift=0pt] {\footnotesize $h$};
\draw [decorate,decoration={brace,amplitude=2.5pt,mirror},xshift=0pt,yshift=-2pt] (9,2) -- (9.97,2) node [black,midway,xshift=0pt,yshift=-7.5pt] {\footnotesize $h$};
\draw (1.75,2.5) node {.};
\draw (2,2.75) node {\footnotesize $P$};
\begin{scope}[shift={(12,0)}]
\fill [gray, opacity=0.2] (0,0) -- (0,3) -- (1.25,4.75) -- (11.25,4.75) -- (11.25,2.75) -- (10,1) -- (1,0);
\draw (0,0) -- (0,3) -- (10,3) -- (10,1) -- (1,0) -- (0,0);
\draw (10,3) -- (11.25,4.75);	
\draw (10,1) -- (11.25,2.75);
\draw (0,3) --(1.25,4.75);
\draw [dashed] (1.25,1.75) -- (1.25,4.75);
\draw (1.25,4.75) -- (11.25,4.75) -- (11.25,2.75);
\draw [dashed] (0,0) -- (1.25,1.75);
\draw [dashed] (1,0) -- (2.25,1.75);
\draw[dashed] (1.25,1.75) -- (2.25,1.75) -- (11.25,2.75);
\draw[dashed] (1,3) -- (2.25,4.75);
\fill[gray,opacity=0.2] (1,-1) -- (1,4) -- (2.25,5.75) -- (2.25,0.5) -- (1,-1);
\draw[dashed] (1,-1) -- (1,4) -- (2.25,5.75) -- (2.25,0.5) -- (1,-1);
\draw (3,4.65) node [right,above] {\footnotesize $\de H_j$};
\draw [<-] (0.5,5) -- (1.5,5);
\draw (1,5.4) node {\footnotesize $\nu_j$};
\end{scope}
\end{tikzpicture}
\caption{The set $D\subset\R^3$ (on the left) and its convex hull (on the right).}
\label{fig:D_shape}
\end{figure}
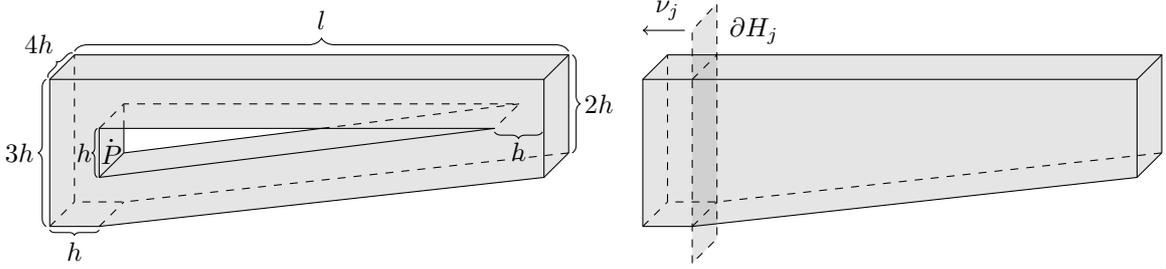

\begin{example}[The set $D$ in $\R^3$]
\label{exa:D_shape}
Let $l>2h>0$ and consider the set $D\subset\R^3$ in \cref{fig:D_shape}.
We can compute
\begin{align*}
\Haus^2(\de D)
&=
12 lh
+
4 h \sqrt{(l-h)^2+h^2}
+
4 h \sqrt{(l-2h)^2+h^2}
+
23 h^2,
\\
\Haus^{2}(\de(\co(D)))
&=
9 lh
+
4 h \sqrt{(l-h)^2+h^2}
+
25 h^2,
\\
\diam(\co(D))
&=
\sqrt{l^2+25 h^2}.
\end{align*}
Since $D$ is not convex, we must have that $k_{\min}(D)\ge2$, and a simple geometric argument allows to conclude that $k_{\min}(D)=3$. 
From~\eqref{eq:leonetti_deficit} we deduce that
\begin{equation*}
k_{\min}(D)\ge
\ceil*{
\frac{\Haus^{2}(\de D)}{\Haus^{2}(\de(\co(D)))}
}
=2,
\end{equation*}
since an elementary computation shows that
\begin{equation*}
\frac{\Haus^{2}(\de D)}{\Haus^{2}(\de(\co(D)))}
=
\frac{
12 l
+
4 \sqrt{(l-h)^2+h^2}
+
4 \sqrt{(l-2h)^2+h^2}
+
23 h
}{
9 l
+
4 \sqrt{(l-h)^2+h^2}
+
25 h
}
\in(1,2)
\end{equation*}
whenever $l>2h>0$.
We now consider the point $P\in\de D$ as shown in \cref{fig:D_shape}.
For every decomposition of $D$ into convex bodies, there exists a convex body $E_j$ containing~$P$.
Since $E_j$ is convex and contained in $D$, we must have that $E_j\subset H_j$, where $H_j$ is the half-space such that $\de H_j$ contains the face of $D$ to which the point $P$ belongs, see \cref{fig:D_shape}. 
Consequently, we must have
\begin{equation*}
h(\co(D)\cap H_j,\co(D))
=
l-h,
\quad
\Haus^2(\co(D)\cap\de H_j)
=
12 h^2,
\quad
\rho_{\nu_j}(\co(D))
=
\sqrt{\frac{12 h^2}\pi},
\end{equation*}
where $\nu_j\in\mathbb S^2$ is the inner unit normal of the half-space $H_j$ as in \cref{fig:D_shape}.
We now let $l>0$ be fixed.
We apply \cref{res:main} with
\begin{align*}
p=1,
\quad
\alpha
=
\frac{l-h}{\sqrt{l^2+25h^2}},
\quad
\beta=1.
\end{align*}
Provided that we choose $h\in\left(0,\frac l2\right)$ sufficiently small, we conclude that
\begin{align*}
k_{\min}(D)
&\ge
\ceil*{
\frac{
\Haus^2(\de D)
+
\pi\left(\frac{l-h}{\sqrt{l^2+25 h^2}}\right)^2
\frac{
\sqrt{\frac{12 h^2}{\pi}}
\,
\left(\sqrt{l^2+25 h^2}\right)^2
}{
\sqrt{\frac{12 h^2}{\pi}}
+
\sqrt{\frac{12 h^2}{\pi}
+
\left(\frac{l-h}{\sqrt{l^2+25 h^2}}\right)^2
\left(\sqrt{l^2+25 h^2}\right)^2
}
}
}{
\Haus^2(\de(\co(D)))
}
}
\\
&=
\ceil*{
\frac{
12 l
+
4 \sqrt{(l-h)^2+h^2}
+
4 \sqrt{(l-2h)^2+h^2}
+
23 h
+
\frac{
\sqrt{12\pi}\,(l-h)^2
}{
\sqrt{\frac{12 h^2}{\pi}}
+
\sqrt{\frac{12 h^2}{\pi}
+
(l-h)^2
}
}
}{
9 l
+
4 \sqrt{(l-h)^2+h^2}
+
25 h
}
}
=3,
\end{align*}
since 
\begin{align*}
\lim_{h\to0^+}
&\frac{
12 l
+
4 \sqrt{(l-h)^2+h^2}
+
4 \sqrt{(l-2h)^2+h^2}
+
23 h
+
\frac{
\sqrt{12\pi}\,(l-h)^2
}{
\sqrt{\frac{12 h^2}{\pi}}
+
\sqrt{\frac{12 h^2}{\pi}
+
(l-h)^2
}
}
}{
9 l
+
4 \sqrt{(l-h)^2+h^2}
+
25 h
}
\\
&=
\frac{20+\sqrt{12\pi}}{13}
\in(2,3).
\end{align*}
\end{example}

\begin{figure}[H]
\begin{tikzpicture}[scale=0.65]
\fill [gray, opacity=0.2] (0,0) -- (0,2) -- (0.5,2.5) -- (1.5,2.5) -- (1.5,1.5) -- (10.5,1.5) -- (10.5,0.5) -- (10,0);
\fill [gray, opacity=0.2] (9,1.5) -- (9,2) -- (9.5,2.5) -- (10.5,2.5) -- (10.5,1.5);
\draw (0,0) -- (0,2) -- (1,2) -- (1,1) -- (9,1) -- (9,2) -- (10,2) -- (10,0) -- (0,0);
\draw (0,2) -- (0.5,2.5) -- (1.5,2.5) -- (1.5,1.5) -- (9,1.5);
\draw (1,1) -- (1.5,1.5);
\draw (1,2) -- (1.5,2.5);
\draw (9,2) -- (9.5,2.5) -- (10.5,2.5) -- (10.5,0.5) -- (10,0);
\draw (10,2) -- (10.5,2.5);
\draw [dashed] (0,0) -- (0.5,0.5) -- (0.5,2.5);
\draw [dashed] (0.5,0.5) -- (10.5,0.5);
\draw [dashed] (9.5,2.5) -- (9.5,1.5) -- (9,1);
\draw [dashed] (9,1.5) -- (9.5,1.5);
\draw [decorate,decoration={brace,amplitude=3pt},xshift=-2pt,yshift=0pt] (0,0) -- (0,2) node [black,midway,xshift=-10pt] {\footnotesize $2h$};
\draw [decorate,decoration={brace,amplitude=2pt},xshift=-1pt,yshift=1pt] (0,2) -- (0.5,2.5) node [black,midway,xshift=-8pt,yshift=5pt] {\footnotesize $h$};
\draw [decorate,decoration={brace,amplitude=2pt},xshift=0pt,yshift=2pt] (0.5,2.5) -- (1.5,2.5) node [black,midway,xshift=0pt,yshift=8pt] {\footnotesize $h$};
\draw [decorate,decoration={brace,amplitude=5pt,mirror},xshift=0pt,yshift=-2pt] (0,0) -- (10,0) node [black,midway,xshift=0pt,yshift=-12pt] {\footnotesize $l$};
\draw [decorate,decoration={brace,amplitude=2pt,mirror},xshift=2pt,yshift=0pt] (1.5,1.5) -- (1.5,2.5) node [black,midway,xshift=6pt,yshift=0pt] {\footnotesize $h$};
\draw (1.25,1.75) node {.};
\draw (1.25,1.45) node {\footnotesize $P$};
\draw (9.25,1.75) node {.};
\draw (9,1.45) node {\footnotesize $Q$};
\begin{scope}[shift={(12,0)}]
\fill [gray, opacity=0.2] (0,0) -- (0,2) -- (0.5,2.5) -- (10.5,2.5) -- (10.5,0.5) -- (10,0);
\fill [gray, opacity=0.2] (1,-1) -- (1,3) -- (1.5,3.5) -- (1.5,-0.5);
\fill [gray, opacity=0.2] (9,-1) -- (9,3) -- (9.5,3.5) -- (9.5,-0.5);
\draw (0,0) -- (0,2) -- (10,2) -- (10,0) -- (0,0);
\draw (0,2) -- (0.5,2.5) -- (10.5,2.5) -- (10.5,0.5) -- (10,0);
\draw (10,2) -- (10.5,2.5);
\draw [dashed] (0,0) -- (0.5,0.5) -- (10.5,0.5);
\draw [dashed] (0.5,0.5) -- (0.5,2.5);
\draw [dashed] (1,-1) -- (1,3) -- (1.5,3.5) -- (1.5,-0.5) -- (1,-1);
\draw [dashed] (1,0) -- (1.5,0.5);
\draw [dashed] (1,2) -- (1.5,2.5);
\draw [dashed] (9,-1) -- (9,3) -- (9.5,3.5) -- (9.5,-0.5) -- (9,-1);
\draw [dashed] (9,0) -- (9.5,0.5);
\draw [dashed] (9,2) -- (9.5,2.5);
\draw (2.25,2.5) node [right,above] {\footnotesize $\de H_j$};
\draw [<-] (0,3) -- (0.85,3);
\draw (0.5,3.4) node {\footnotesize $\nu_j$};
\end{scope}
\end{tikzpicture}
\caption{The set $U\subset\R^3$ (on the left) and its convex hull (on the right).}
\label{fig:U_shape}
\end{figure}
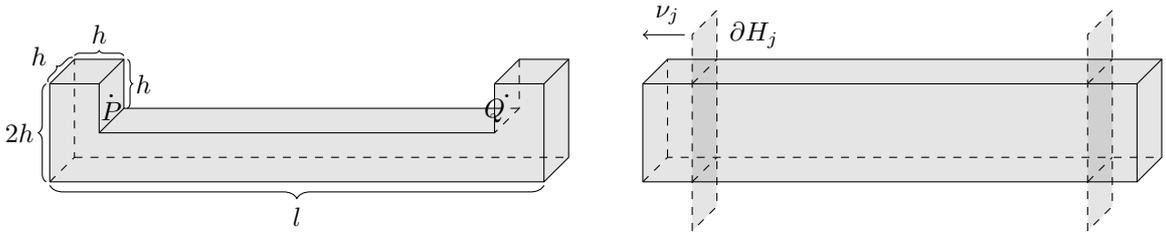

\begin{example}[The set $U$ in $\R^3$]
Let $l>3h>0$ and consider the set $U\subset\R^3$ in \cref{fig:U_shape}.
We can compute
\begin{align*}
\Haus^2(\de U)
&=
4hl+10h^2,\ 
\Haus^2(\de(\co(U)))
=
6hl+4h^2,\ 
\diam(\co(U))
=
\sqrt{l^2+5h^2}.
\end{align*}
Since $U$ is not convex, we must have that $k_{\min}(U)\ge2$, and a simple geometric argument allows to conclude that $k_{\min}(U)=3$.
From~\eqref{eq:leonetti_deficit} we deduce that
\begin{equation*}
k_{\min}(U)\ge
\ceil*{
\frac{\Haus^{2}(\de U)}{\Haus^{2}(\de(\co(U)))}
}=1,
\end{equation*}
since an elementary computation shows that
\begin{equation*}
\frac{\Haus^{2}(\de U)}{\Haus^{2}(\de(\co(U)))}
=
\frac{
4l+10h
}{
6l+4h
}
\in(0,1)
\end{equation*}
whenever $l>3h>0$.
We now consider the points $P,Q\in\de U$ as shown in \cref{fig:U_shape}. 
For every decomposition of $U$ into convex bodies, there exists two convex bodies $E_j$ and $E_k$ containing~$P$ and~$Q$ respectively.
Since the segment $PQ$ is not contained in $U$, it follows that $E_j$ cannot contain~$Q$. 
Since $E_j$ is convex and contained in $U$, we must have that $E_j\subset H_j$, where $H_j$ is the half-space such that $\de H_j$ contains the face of $U$ to which the point $P$ belongs, see \cref{fig:U_shape}.
Consequently, we must have
\begin{align*}
h(\co(U)\cap H_j,\co(U))=l-h,
\quad
\Haus^2(\co(U)\cap\de H_j)=2h^2,
\quad
\rho_{\nu_j}(\co(U))=\sqrt{\frac{2h^2}{\pi}},	
\end{align*}
where $\nu_j\in\mathbb S^2$ is the inner unit normal of  the half-space $H_j$ as in \cref{fig:U_shape}.
By the symmetry of~$U$,  a similar argument can be used for the convex component $E_k$ containing~$Q$.
We now let $l>0$ be fixed.
We apply \cref{res:main} with
\begin{align*}
p=2,
\quad
\alpha=\frac{l-h}{\sqrt{l^2+5h^2}},
\quad
\beta=1.
\end{align*}
Provided that we choose $h\in(0,\frac l3)$ sufficiently small, we conclude that
\begin{align*}
k_{\min}(U)
&\ge
\ceil*{
\frac{
4hl+10h^2
+
2\pi\left(\tfrac{l-h}{\sqrt{l^2+5h^2}}\right)^2
\tfrac{
\sqrt{\tfrac{2h^2}{\pi}}
\,
\left(\displaystyle\sqrt{l^2+5h^2}\right)^2
}
{
\sqrt{\tfrac{2h^2}{\pi}}
+
\sqrt{\tfrac{2h^2}{\pi}
+
\left(\tfrac{l-h}{\sqrt{l^2+5h^2}}\right)^2
\,
\left(\displaystyle\sqrt{l^2+5h^2}\right)^2}}
}
{6hl+4h^2}
}
\\
&=
\ceil*{
\frac{
4l+10h
+
\tfrac{2\sqrt{2\pi}\,(l-h)^2}
{
\sqrt{\tfrac{2h^2}{\pi}}
+
\sqrt{\tfrac{2h^2}{\pi}
+
(l-h)^2}
}
}
{6l+4h}
}
=2,
\end{align*} 
since
\begin{equation*}
\lim_{h\to0^+}
\frac{
4l+10h
+
\tfrac{2\sqrt{2\pi}\,(l-h)^2}
{
\sqrt{\tfrac{2h^2}{\pi}}
+
\sqrt{\tfrac{2h^2}{\pi}
+
(l-h)^2}
}
}
{6l+4h}
=
\frac{4+2\sqrt{2\pi}}{6}
\in(1,2).
\end{equation*}
The above computations prove that, in this case, although the lower bound given by~\eqref{eq:main_k_est} is strictly better than the one given by~\eqref{eq:leonetti_deficit}, the inequality~\eqref{eq:main_k_est} is not sharp.    
\end{example}

\subsection{An example in \texorpdfstring{$\R^n$}{Rˆn}}

We conclude this section with \cref{exa:L_23_shape} below, showing that for all $n\ge3$ our \cref{res:main} provides an improvement of the inequality~\eqref{eq:leonetti_deficit} established in~\cite{LL08}.
In \cref{exa:L_23_shape} we will need to apply the following result, whose elementary proof is detailed below for the reader's convenience.

\begin{lemma}\label{res:cilindro}
Let $\ell\in(0,+\infty)$ and let $Q\subset\R^2$ be a set with 
\begin{equation*}
\Haus^1(\de Q)<+\infty
\quad\text{and}\quad 
\Haus^2(Q)<+\infty.
\end{equation*}
If 
$E_n=Q\times[0,\ell]^{n-2}\subset\R^n$, then
\begin{equation}
\label{eq:coarea}
\Haus^{n-1}(\de E_n)
=
\ell^{n-2}\,\Haus^1(\de Q)
+
2(n-2)\,\ell^{n-3}\,\Haus^2(Q)
\end{equation}	
for all $n\ge2$.
\end{lemma}

\begin{proof}
By definition, the set $E_n\subset\R^n$ satisfies
\begin{equation}
\label{eq:vol_prod}
\Haus^n(E_n)=\ell^{n-2}\,\Haus^2(Q).
\end{equation}
Moreover, since we can recursively write $E_n=E_{n-1}\times[0,\ell]$ and thus
\begin{equation*}
\de E_n 
=
((\de E_{n-1})\times [0,\ell])
\,\cup\,
(E_{n-1}\times\set*{0,\ell}),
\end{equation*}
by the coarea formula we can compute
\begin{equation*}
\Haus^{n-1}(\de E_n)
=
2\Haus^{n-1}(E_{n-1})
+
\ell\,\Haus^{n-2}(\de E_{n-1})
\end{equation*}
for all $n\ge2$.
The validity of~\eqref{eq:coarea} can thus be checked by induction, thanks to~\eqref{eq:vol_prod}.
\end{proof}

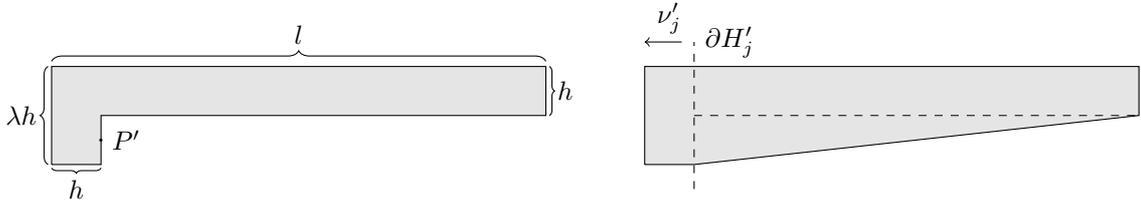
\begin{figure}[H]
\begin{tikzpicture}[scale=0.65]
\fill [gray, opacity=0.2] (0,0) -- (0,2) -- (10,2) -- (10,1) -- (1,1) -- (1,0) -- (0,0);
\draw (0,0) -- (0,2) -- (10,2) -- (10,1) -- (1,1) -- (1,0) -- (0,0); 
\draw (1,0.5) node {.};
\draw (1,0.5) node [right] {\footnotesize $P'$};
\draw [decorate,decoration={brace,amplitude=3pt},xshift=-2pt,yshift=0pt] (0,0) -- (0,2) node [black,midway,xshift=-10pt] {\footnotesize $\lambda h$};
\draw [decorate,decoration={brace,amplitude=2pt,mirror},xshift=0pt,yshift=-2pt] (0,0) -- (1,0) node [black,midway,yshift=-8pt] {\footnotesize $h$};
\draw [decorate,decoration={brace,amplitude=2pt,mirror},xshift=0pt,xshift=2pt] (10,1) -- (10,2) node [black,midway,xshift=6pt] {\footnotesize $h$};
\draw [decorate,decoration={brace,amplitude=4pt},xshift=0pt,yshift=3pt] (0,2) -- (10,2) node [black,midway,yshift=10pt] {\footnotesize $l$};
\begin{scope}[shift={(12,0)}]
\fill [gray, opacity=0.2] (0,0) -- (0,2) -- (10,2) -- (10,1) -- (1,0) -- (0,0);
\draw (0,0) -- (0,2) -- (10,2) -- (10,1) -- (1,0) -- (0,0);
\draw [dashed] (1,-0.5) -- (1,2.5);
\draw (1,2.5) node [right] {\footnotesize $\de H_j'$};
\draw [<-] (0,2.5) -- (.75,2.5);
\draw (.5,2.45) node [above] {\footnotesize $\nu_j'$};
\draw [dashed, very thin] (1,1) -- (10,1);
\end{scope}
\end{tikzpicture}
\caption{The body $L_2\subset\R^2$ (on the left) and its convex hull (on the right).}
\label{fig:L_23_shapes}
\end{figure}

\begin{example}[The set $L_n\subset\R^n$ for $n\ge3$]
\label{exa:L_23_shape}
Let $l>h>0$ and $\lambda>1$ and consider the set $L_n=L_2\times[0,h]^{n-2}\subset\R^n$ for $n\ge3$, where $L_2\subset\R^n$ is the set in \cref{fig:L_23_shapes}.
Note that
\begin{align*}
\Haus^1(\de L_2)
=
2l+2\lambda h,
\quad
\Haus^2(L_2)
=
h\,\big(l+(\lambda-1)h\big)
\end{align*}
and, similarly,
\begin{align*}
\Haus^1(\de(\co(L_2)))
&=
l+\sqrt{(l-h)^2+(\lambda-1)^2 h^2}+(\lambda+2)h,
\\
\Haus^2(\co(L_2))
&=\frac h2\,\big((\lambda+1)l+(\lambda-1)h\big).
\end{align*}
Since 
$\co(L_n)=\co(L_2)\times[0,h]^{n-2}$,
we can apply \cref{res:cilindro} to compute
\begin{align*}
\Haus^{n-1}(\de L_n)
&=
2h^{n-2}\big((n-1)l+((n-1)\lambda-n+2))h\big),
\\
\Haus^{n-1}(\de(\co(L_n)))
&=
h^{n-2}\big(((n-2)\lambda+n-1)l+\sqrt{(l-h)^2+(\lambda-1)^2h^2}
\\
&\quad+
((n-1)\lambda-n+4)h\big),
\\
\diam(\co(L_n))
&=
\sqrt{l^2+(\lambda^2+n-2)h^2}
\end{align*}
for all $n\ge3$.
Note that $L_n$ is not convex, so we must have that $k_{\min}(L_n)\ge2$ for all~$n\ge3$.
In fact, a simple geometric decomposition proves that $k_{\min}(L_n)=2$ for all~$n\ge3$.
We now consider the point $P=(P',0)\in L_n$, where $P'\in\de L_2$ is shown in \cref{fig:L_23_shapes}.
For every decomposition of $L_n$ into convex bodies, there exists a convex body $E_j$ containing~$P$.
Since $E_j$ is convex and contained in $L_n$, we must have that its projection $E_j'=\mathsf{P}_{\R^2}(E_j)$ is a convex body contained in $L_2\cap H_j'$, where $\mathsf{P}_{\R^2}\colon\R^n\to\R^2$ is the canonical projection onto the first two coordinates and $H_j'$ is the half-plane such that $\de H_j'$ contains the face of $L_2$ to which the point $P$ belongs, see \cref{fig:L_23_shapes}.
Therefore, we must have that 
$E_j\subset H_j$,
where $H_j$ is the half-space $H_j=\mathsf{P}_{\R^2}^{-1}(H_j')\subset\R^n$. 
Consequently, we must have 
\begin{equation*}
h(\co(L_n)\cap H_j,\co(L_n))
=
l-h,
\quad
\Haus^{n-1}(\co(L_n)\cap\de H_j)
=
\lambda h^{n-1},
\quad
\rho_{\nu_j}(\co(L_n))=\sqrt[n-1]{\frac{\lambda h^{n-1}}{\omega_{n-1}}},
\end{equation*} 
where $\nu_j\in\mathbb S^{n-1}$ is the inner unit normal of the half-space $H_j$ (precisely, $\nu_j=(\nu_j',0)$, where $\nu_j'$ is the inner unit normal of $H_j'$, see \cref{fig:L_23_shapes}).
We now let $l>0$ be fixed.
We apply \cref{res:main} with
\begin{equation*}
p=1,
\quad
\alpha=\frac{l-h}{\sqrt{l^2+(\lambda^2+n-2)h^2}},
\quad
\beta=1.
\end{equation*} 
We are going to choose $\lambda>1$ as a dimensional constant and $h\in(0,l)$ sufficiently small. 
Indeed, for any given $\lambda>1$, we have that 
\begin{align*}
\lim_{h\to0^+}
\frac{
\Haus^{n-1}(\de L_n)
}{
\Haus^{n-1}(\de(\co(L_n)))
}
=
\frac{
2n-2
}{
(n-2)\lambda+n
}
\end{align*}
and, similarly,
\begin{align*}
\lim_{h\to0^+}
\frac{
\Haus^{n-1}(\de L_n)
+
\omega_{n-1}\alpha^2\beta^{\frac{n-2}{n-1}}
\tfrac{\rho_{\nu_j}(\co(L_n))^{n-2}\diam(\co(L_n))^2}{\rho_{\nu_j}(\co(L_n))+\sqrt{\rho_{\nu_j}(\co(E))^2+\alpha^2\diam(\co(L_n))^2}}
}{
\Haus^{n-1}(\de(\co(L_n)))
}
=
\frac{
2n-2
+c_n\lambda^{\frac{n-2}{n-1}}
}{
(n-2)\lambda+n
},
\end{align*}
where $c_n=\omega_{n-1}^{\frac1{n-1}}>0$ is a dimensional constant.
Since $\lambda>1$, we have that
\begin{equation*}
\frac{
2n-2
}{
(n-2)\lambda+n
}
<1
\quad
\text{for all}\ n\ge3.
\end{equation*}
On the other hand, we obviously have
\begin{equation*}
\frac{
2n-2
+c_n\lambda^{\frac{n-2}{n-1}}
}{
(n-2)\lambda+n
}
>1
\iff
\lambda^{\frac{n-2}{n-1}}
>\frac{n-2}{c_n}(\lambda-1)
\end{equation*}
and it is possible to verify that the last inequality admits solutions in the interval $(1,+\infty)$.
Consequently, for each $n\ge3$ we can find $\lambda_n\in(1,+\infty)$ such that
\begin{equation*}
\frac{
2n-2
+c_n\lambda_n^{\frac{n-2}{n-1}}
}{
(n-2)\lambda_n+n
}
>1.
\end{equation*}
Therefore, provided that we choose $\lambda=\lambda_n$ as above and $h\in(0,l)$ sufficiently small, we conclude that the set $L_n\subset\R^n$ corresponding to these choices of parameters satisfies
\begin{equation*}
\ceil*{
\frac{
\Haus^{n-1}(\de L_n)
}{
\Haus^{n-1}(\de(\co(L_n)))
}
}
=1
\end{equation*}
and
\begin{equation*}
\ceil*{
\frac{
\Haus^{n-1}(\de L_n)
+
\omega_{n-1}\alpha^2\beta^{\frac{n-2}{n-1}}
\tfrac{\rho_{\nu_{i_j}}(\co(L_n))^{n-2}\diam(\co(L_n))^2}{\rho_{\nu_j}(\co(L_n))+\sqrt{\rho_{\nu_j}(\co(E))^2+\alpha^2\diam(\co(L_n))^2}}
}{
\Haus^{n-1}(\de(\co(L_n)))
}
}
=2.
\end{equation*}
\end{example}


\end{document}